\documentclass[11pt]{amsart}

\usepackage{amsmath,amssymb,amsthm}
\usepackage{mathtools}
\usepackage[margin=1.2in]{geometry}
\usepackage[dvipsnames]{xcolor}
\usepackage[colorlinks=true,linkcolor=MidnightBlue,citecolor=MidnightBlue,urlcolor=MidnightBlue]{hyperref}

\numberwithin{equation}{section}

\newtheorem{theorem}{Theorem}[section]
\newtheorem{lemma}[theorem]{Lemma}
\newtheorem{proposition}[theorem]{Proposition}
\newtheorem{corollary}[theorem]{Corollary}
\newtheorem{question}[theorem]{Question}
\theoremstyle{definition}
\newtheorem{definition}[theorem]{Definition}
\newtheorem*{theorem*}{Theorem}

\theoremstyle{definition}
\newtheorem{example}{Example}
\theoremstyle{remark}
\newtheorem{remark}{Remark}

\theoremstyle{remark}
\newtheorem*{open}{Open Problems}

\newcommand{\NSym}{\mathbf{NSym}}
\newcommand{\QSym}{\mathrm{QSym}}
\newcommand{\Sym}{\Lambda}
\newcommand{\rc}{\operatorname{rc}}
\newcommand{\rp}{\operatorname{rp}}
\newcommand{\rec}{\operatorname{rec}}
\newcommand{\Alt}{\mathcal{A}}
\newcommand{\bA}{\mathbf{A}}
\newcommand{\bR}{\mathbf{R}}
\newcommand{\Q}{\mathbb{Q}}
\newcommand{\Z}{\mathbb{Z}}
\DeclareMathOperator{\piu}{\pi_u}

\title[Record compositions of alternating permutations]%
{Record compositions of alternating permutations and noncommutative symmetric functions}

\author{Evan Chen}
\email{evan@axiommath.ai}
\author{Michal Mogielnicki}
\email{michal@axiommath.ai}
\author{Ken Ono}
\email{ken@axiommath.ai}

\address{Axiom Math, 124 University Avenue, Palo Alto, CA 94301, USA}


\begin{document}

\begin{abstract}
  Amdeberhan, Shareshian, and Stanley recently proved that a
  function $\varphi$ arising in the theory of partition
  Eisenstein series counts the alternating permutations of
  $\{1,\dots,2n\}$ with a given ``record'' partition, and they asked whether
  there is a similar theory for record \emph{compositions}, suggesting a role for
  noncommutative symmetric functions.  Here we solve their open problem
  by showing that the number of alternating permutations of
  $\{1,\dots,2n\}$ with record composition
  $(\alpha_1,\dots,\alpha_\ell)$ is
  \[ \prod_{j=1}^{\ell}\binom{2s_j-1}{2\alpha_j-1}E_{2\alpha_j-1}, \]
  where $s_j=\alpha_1+\dots+\alpha_j$, $E_k$ is an Euler number, and the
  record composition of $w=a_1a_2\dots a_{2n}$ (so $a_1>a_2<a_3>\dotsb$)
  lists the factor lengths obtained by cutting $a_1a_3\dots a_{2n-1}$
  before each left-to-right maximum other than the first.
  These numbers are the coefficients of a natural lift
  of the degree-$n$ sprout symmetric function with seed $\sec(\sqrt{t}\,)$
  to noncommutative symmetric functions, expanded in products of noncommutative power sums of the first kind.
  An analogous refinement holds for every sprout sequence whose seed is given by the
  exponential formula.
  AxiomProver autonomously produced and verified the results in this paper in Lean.
\end{abstract}

\maketitle

\section{Introduction}\label{sec:intro}
\subsection{Background}
Given a partition $\lambda\vdash n$, write $m_k=m_k(\lambda)$ for the
number of parts of $\lambda$ equal to $k$ (so $\sum_k k\,m_k=n$).
In \cite{AOS}, the third author, together with Amdeberhan and Singh,
defined the functions
\[ \varphi(\lambda) = 4^n (2n)! \prod_{k=1}^n \frac{1}{m_k!}
  \left( \frac{(4^k-1) B_{2k}}{(2k) (2k)!} \right)^{m_k} \in \mathbb Z \]
where $B_{2k}$ denotes the Bernoulli number.
This function arose in the context of expressing
derivatives of Ramanujan's theta functions as traces of partition Eisenstein series.

In this paper we will only be concerned with the absolute value $|\varphi(\lambda)|$.
By \cite[equation (3.3)]{AOS}, $|\varphi(\lambda)|$
may be rewritten in terms of the \emph{Euler numbers} $E_k$ as
\begin{equation}\label{eq:phiintro}
  |\varphi(\lambda)|\;=\;(2n)!\,\prod_{k=1}^{n}\frac{1}{m_k!}
  \left(\frac{E_{2k-1}}{(2k)!}\right)^{\!m_k}.
\end{equation}
For example, $|\varphi(2,2,1)| = 10!\cdot\frac{1}{1!}(E_1/2!)^{1}
\cdot\frac{1}{2!}(E_3/4!)^{2} = 3628800\cdot\frac12\cdot(1/2)\cdot(2/24)^2 = 6300$.
For this paper we can take \eqref{eq:phiintro} as our working definition.
Note that from \eqref{eq:phiintro},
it is possible to show directly (see \cite[equation (3.1)]{ASS}) that:
\begin{equation}
  \sum_{\lambda \vdash n} |\varphi(\lambda)| = E_{2n}.
  \label{eq:rp_total}
\end{equation}

\subsection{The function $|\varphi(\lambda)|$ counts alternating permutations by record partition.}
In 2026, Amdeberhan, Shareshian, and Stanley \cite{ASS}
produced a combinatorial interpretation of \eqref{eq:rp_total},
showing that $|\varphi(\lambda)|$ provides a count of a certain class of alternating permutations.
We describe this interpretation now.

Recall that a permutation $w=a_1a_2\dots a_m$ of a set of numbers,
written in one-line notation,
is called \emph{alternating} if $a_1>a_2<a_3>a_4<\dotsb$.
For example, $3\,1\,4\,2$ and $5\,2\,4\,1\,3$ are alternating, while
$1\,3\,2\,4$ is not.
We let $\Alt_{2n}$ denote the set of alternating permutations
of $[2n] \coloneqq \{1, 2, \dots, 2n\}$.
It is well known that the number of alternating permutations
is given by the Euler number; that is, $\#\Alt_{2n} = E_{2n}$.

We attach two combinatorial statistics to each $w = a_1a_2\dots a_{2n}\in\Alt_{2n}$.
Define the \emph{odd-indexed subword}
\[ \hat w\;\coloneqq \;a_1a_3a_5\dots a_{2n-1}. \]
Recall that, in a sequence of distinct numbers,
an entry is a \emph{record} (or \emph{left-to-right maximum})
if it is larger than every entry before it.
(The very first entry is always a record.)
Cutting $\hat w$ immediately before each record (other than the first)
breaks $\hat w$ into consecutive blocks.
Then:
\begin{itemize}
  \item The ordered sequence of block lengths is called the \emph{record composition},
  and we denote it by $\rc(\hat w)$.
  \item We then define the \emph{record partition} $\rp(\hat w)$
  by sorting the entries of $\rc(\hat w)$ in weakly decreasing order.
\end{itemize}

\begin{example}
  [{\cite[Example 3.1]{ASS}}]
  \label{ex:intro}
  Take $w=7\,2\,5\,4\,8\,3\,10\,6\,9\,1\in\Alt_{10}$, so $n=5$.
  The odd-indexed subword is
  \[ \hat w=a_1a_3a_5a_7a_9=7\,5\,8\,10\,9.  \]
  Reading $\hat w$ from left to right,
  the records are $7$ at position $1$,
  then $8$ at position $3$, then $10$ at position $4$.
  (The entry $5$ at position $2$ is not a record because $5<7$,
  and $9$ at position $5$ is not a record because $9<10$.)
  Cutting $\hat w$ before the records at positions $3$ and $4$ gives the three blocks
  $7\,5\mid 8\mid 10\,9$, of lengths $2$, $1$, $2$.
  Hence
  \[ \rc(\hat w)=(2,1,2) \quad\text{and}\quad \rp(\hat w)=(2,2,1). \]
\end{example}

The combinatorial interpretation promised is then:
\begin{theorem*}
  [{\cite[Theorem 3.2]{ASS}}]
  The quantity $|\varphi(\lambda)|$ counts alternating permutations
  with record partition $\lambda$; that is,
  \begin{equation}
    |\varphi(\lambda)| = \# \left\{ w \in \Alt_{2n} \mid \rp(\hat w) = \lambda \right\}.
    \label{eq:ass32}
  \end{equation}
\end{theorem*}
Note that summing \eqref{eq:ass32}
across all partitions $\lambda$ immediately recovers \eqref{eq:rp_total}.

\subsection{Connection to symmetric functions via sprout sequences}
The preceding theorem does not make any reference to symmetric functions.
However, the paper \cite{ASS} is actually concerned with \emph{sprout sequences}.
A sprout sequence $\mathfrak{R} = (R_0=1,R_1,R_2,\dots)$,
with $R_n$ symmetric and homogeneous of degree $n$ in the variables $x=(x_1,x_2,\dots)$,
is generated from a power series $F(t)=1+a_1t+a_2t^2+\dotsb$ (the \emph{seed})
by the rule
\[ \sum_{n\ge 0}R_nt^n\;=\;\prod_{i\ge 1}F(x_it). \]
To put \cite[Theorem 3.2]{ASS} in this context, Amdeberhan, Shareshian, and Stanley
also prove that the sprout sequence with seed $F(t) = \sec \sqrt t$
recovers $|\varphi(\lambda)|$.

\begin{theorem*}
  [{\cite[Theorem 4.1]{ASS}}]
  Let $(A_0,A_1,A_2,\dots)$ denote the sprout sequence generated by the seed
  \[ F(t)=\sec(\sqrt{t}). \]
  Then we have
  \begin{equation}
    A_n = \frac{1}{(2n)!} \sum_{\lambda \vdash n} |\varphi(\lambda)| p_\lambda
    \label{eq:ass41}
  \end{equation}
  where $p_\lambda$ is a power-sum symmetric function.
\end{theorem*}

\subsection{Our refinement to counting by record compositions}
The previously stated results are for record partitions,
in the context of symmetric functions.
In {\cite[\S3]{ASS}}, Amdeberhan--Shareshian--Stanley pose the following open problem
asking whether these results can be refined to record compositions:

\begin{open}
  [Amdeberhan--Shareshian--Stanley, {\cite[\S3]{ASS}}]
  (1) Is there a refinement of \eqref{eq:ass32} to record \emph{compositions}?
  In other words, is there a function $N(\alpha)$
  on compositions $\alpha$ of $n$ such that
  \[
  N(\alpha)\;=\;\#\{\,w\in\Alt_{2n} : \rc(\hat w)=\alpha\,\}
  \qquad\text{and}\qquad
  \sum_{\alpha\,\sim\,\lambda} N(\alpha)\;=\;|\varphi(\lambda)|,
  \]
  the second sum being over all compositions $\alpha$ that are rearrangements of $\lambda$?

  (2) Moreover, since compositions play the role of
  partitions in the theory of noncommutative symmetric functions~\cite{GKLLRT},
  is there a noncommutative analog of the sprout sequence in \eqref{eq:ass41}?
\end{open}

The purpose of this paper is to answer both questions affirmatively and explicitly.
To state our results, we introduce some notation.
We write $\alpha = (\alpha_1, \dots, \alpha_\ell) \vDash n$ to mean that
$\alpha$ is a \emph{composition} of $n$,
meaning $\alpha_i > 0$ and $\alpha_1 + \dots + \alpha_\ell = n$.
We define the \emph{$j$-th partial sum} as
$s_j=s_j(\alpha)\coloneqq \alpha_1+\dots+\alpha_j$ (so $s_\ell=n$).

Our first main result is the desired refinement.
This result counts alternating permutations
of $[2n]$ by their record composition, exactly.

\begin{theorem}\label{thm:main}
  If $n \ge 1$ and $\alpha=(\alpha_1,\dots,\alpha_\ell) \vDash n$, then we have
  \[
  N(\alpha)\;\coloneqq \;\#\{\,w\in\Alt_{2n} : \rc(\hat w)=\alpha\,\}
  \;=\;\prod_{j=1}^{\ell}\binom{2s_j-1}{2\alpha_j-1}\,E_{2\alpha_j-1}.
  \]
\end{theorem}
\begin{example}
  Continuing Example~\ref{ex:intro}, for $\alpha=(2,1,2)$ we get $(s_1, s_2, s_3) = (2,3,5)$, and
  \[
  N(2,1,2)=\binom{3}{3}E_3\cdot\binom{5}{1}E_1\cdot\binom{9}{3}E_3
  =2\cdot5\cdot168=1680.
  \]
\end{example}
\begin{remark}
  A striking feature is that $N(\alpha)$ genuinely depends on the \emph{order} of the parts:
  for example, one has $N(1,2)=20$ but $N(2,1)=10$.
  This is exactly the sensitivity to order that the passage from partitions to compositions
  is meant to capture, and it is why a naive count keyed only to the
  record partition---as in~\cite{ASS}---cannot see it.
  Summing Theorem~\ref{thm:main} over all rearrangements of a fixed partition
  $\lambda$ collapses this finer information and recovers the original
  result of~\cite{ASS}; we carry this out in Corollary~\ref{cor:recover}.
\end{remark}

The second main result rewrites $N(\alpha)$ in a form that makes the
connection with noncommutative symmetric functions transparent.  Let
\begin{equation}\label{eq:bdef}
  b_k\;\coloneqq \;\frac{k\,E_{2k-1}}{(2k)!}\qquad(k\ge1),
\end{equation}
so that $\log\sec(\sqrt{t}\,)=\sum_{k\ge1}b_k t^k/k$; these are
exactly the coefficients attached to the seed $\sec(\sqrt t\,)$
in~\cite[proof of Theorem~4.1]{ASS}, in terms of which
$A_n=\sum_{\lambda\vdash n}z_\lambda^{-1}b_\lambda p_\lambda$, where
$b_\lambda\coloneqq b_{\lambda_1}b_{\lambda_2}\dots$ and $p_\lambda$ is a
power sum symmetric function.

\begin{theorem}\label{thm:bform}
  With the notation of Theorem~\ref{thm:main}, we have
  \[
  N(\alpha)\;=\;(2n)!\,\prod_{j=1}^{\ell}\frac{b_{\alpha_j}}{s_j}
  \;=\;(2n)!\,\frac{b_{\alpha_1}b_{\alpha_2}\dots b_{\alpha_\ell}}{\piu(\alpha)},
  \qquad\text{where }\ \piu(\alpha)\coloneqq \prod_{j=1}^{\ell}s_j(\alpha).
  \]
\end{theorem}

The quantity $\piu(\alpha)$, the product of the partial sums of
$\alpha$, is precisely the denominator appearing in the expansion, due
to Gelfand, Krob, Lascoux, Leclerc, Retakh, and Thibon
\cite[Proposition~4.5]{GKLLRT}, of the noncommutative complete
symmetric function $S_n$ in the basis of products of noncommutative
power sums of the first kind:
\[ S_n=\sum_{\alpha\vDash n}\Psi^\alpha/\piu(\alpha) \]
in the algebra
$\NSym$ of noncommutative symmetric functions.  This is not a
coincidence.  Since the $\Psi_k$ freely generate $\NSym$
(Corollary~\ref{cor:psibasis}), for any
sequence of scalars $b=(b_1,b_2,\dots)$ there is a unique graded
algebra endomorphism $\phi_b$ of $\NSym$ with
$\phi_b(\Psi_k)=b_k\Psi_k$.

Given any sprout sequence with seed
$F$, writing
\[ \log F(t)=\sum_{n\ge1}b_nt^n/n \]
as in \cite[eq.\
(1.2)]{ASS}, we call $\bR_n\coloneqq \phi_b(S_n)$ the \emph{canonical
noncommutative lift} of $R_n$. Its image under the standard projection
$\chi$ from $\NSym$ to the algebra of symmetric functions is $R_n$
(see Theorem~\ref{thm:ncsprout}).  For the seed $\sec(\sqrt t\,)$, we write
$\bA_n\coloneqq \phi_b(S_n)$ with $b_k$ as in \eqref{eq:bdef}.

Our third main result then says that the record composition statistic computes
exactly the $\Psi$-coordinates of this lift.
This provides the promised analog of \eqref{eq:ass41}.

\begin{theorem}\label{thm:nsym}
  For all $n\ge1$,
  \[
    (2n)!\,\bA_n\;=\;\sum_{\alpha\vDash n} N(\alpha)\,\Psi^\alpha
    \;=\;\sum_{\alpha\vDash n}\#\{w\in\Alt_{2n}:\rc(\hat w)=\alpha\}\, \Psi^\alpha,
  \]
  and $\chi(\bA_n)=A_n$.  In particular, the expansion of
  $(2n)!\,\bA_n$ in the basis $\{\Psi^\alpha\}$ has nonnegative integer
  coefficients with the combinatorial interpretation above.
\end{theorem}

Theorems~\ref{thm:main}--\ref{thm:nsym} thus settle both of the open
problems. Namely, the record composition count $N(\alpha)$ is the
sought refinement, and it is realized as a coordinate in
$\NSym$ exactly as anticipated in~\cite{ASS}.  Two features of the
$\NSym$ picture merit emphasis, and we record them here.

\begin{remark}\label{rem:firstkind}
  The refinement is genuinely a property of the noncommutative power
  sums \emph{of the first kind}.  The analogous rescaling of the power
  sums of the second kind produces coefficients that are symmetric in
  the parts of $\alpha$, and so cannot detect the order-dependence of
  $N(\alpha)$ (for instance $N(1,2)=20\ne10=N(2,1)$), as explained in
  Remark~\ref{rem:phi}.
\end{remark}

\begin{remark}[Limits of positivity]\label{rem:limits}
  By Theorem~\ref{thm:nsym} the coordinates of $(2n)!\,\bA_n$ in the
  basis $\{\Psi^\alpha\}$ are nonnegative integers.  In the complete
  ($S$-word) and ribbon bases of $\NSym$ they remain integers but need
  \emph{not} be nonnegative, the first failures occurring at $n=4$
  (Remark~\ref{rem:negative}).  Consequently a combinatorial
  interpretation of the $h$-expansion of $(2n)!A_n$, sought
  in~\cite[\S5]{ASS}, cannot be obtained simply by refining along the
  $S$-word coordinates of the canonical lift.
\end{remark}

The paper is organized as follows.
Section~\ref{sec:prelim} fixes notation and recalls the background on alternating permutations,
records, and the theorem of~\cite{ASS}.
Section~\ref{sec:count} proves Theorem~\ref{thm:main} by a bijection
that reorganizes the construction of \cite[Theorem~3.2]{ASS}
so that the blocks are kept in their canonical order.
Section~\ref{sec:product} proves Theorem~\ref{thm:bform}
and deduces the original partition count of \cite{ASS} (see Corollary~\ref{cor:recover}).
Section~\ref{sec:nsym} develops the noncommutative picture,
proves Theorem~\ref{thm:nsym}, and establishes the negative results of
Remarks~\ref{rem:firstkind}--\ref{rem:limits}.
Section~\ref{sec:general} isolates the mechanism behind the argument
and proves a general refinement (see Proposition~\ref{prop:expseed}) for
every sprout sequence whose seed is given by the exponential formula
\cite[Example~1.1(i)]{ASS}, of which Theorem~\ref{thm:main} is the
natural ``doubled'' variant.
Section~\ref{sec:examples} collects worked examples and complete tables for $n\le5$.
All finite assertions of the paper have been verified by exact machine
computation, as detailed in Remark~\ref{rem:verification}.
Section~\ref{sec:AI} describes the accompanying Lean/mathlib
formalization, including the AxiomProver protocol, the formal files,
and the verification environment.

\subsection*{Acknowledgments}
We thank Tewodros Amdeberhan, John Shareshian, and Richard Stanley
for posing the stimulating problem in~\cite{ASS}.

\section{Background and notation}\label{sec:prelim}

This section makes precise the notions sketched in the introduction
and fixes the conventions used throughout.  A reader comfortable with
compositions, alternating permutations, and symmetric functions may
skim it and refer back as needed; the one genuinely new definition,
that of the record composition, is given in
Section~\ref{subsec:records}.  Throughout, $[m]=\{1,2,\dots,m\}$.
All permutations are written in one-line
notation.

\subsection{Compositions and partitions}\label{subsec:comp}
A \emph{composition} of $n\ge 0$ is a (possibly empty) sequence
$\alpha=(\alpha_1,\dots,\alpha_\ell)$ of positive integers with
$\alpha_1+\dots+\alpha_\ell=n$; we write $\alpha\vDash n$ and
$\ell(\alpha)=\ell$ for its \emph{length}.  Its \emph{partial sums}
are $s_j=s_j(\alpha)=\alpha_1+\dots+\alpha_j$ for $0\le j\le\ell$
(so $s_0=0$ and $s_\ell=n$), and we set
\[
\piu(\alpha)\;\coloneqq \;\prod_{j=1}^{\ell(\alpha)} s_j(\alpha),
\]
the product of the (positive) partial sums, following the notation
of~\cite[Proposition~4.5]{GKLLRT}; thus $\piu$ of the empty composition is the
empty product $1$.  A \emph{partition} of $n$, written
$\lambda\vdash n$, is a composition whose parts weakly decrease.  We
write $\lambda(\alpha)$ for the partition obtained by sorting the
parts of $\alpha$, and say $\alpha$ is a \emph{rearrangement} of
$\lambda$, written $\alpha\sim\lambda$, if $\lambda(\alpha)=\lambda$.
For a partition $\lambda$ we write $m_k=m_k(\lambda)$ for the number
of parts equal to $k$, and
\[
z_\lambda\;\coloneqq \;\prod_{k\ge1}k^{m_k}\,m_k! \,.
\]

For $n\ge1$, set
$S(\alpha)\coloneqq \{s_1(\alpha),\dots,s_{\ell-1}(\alpha)\}\subseteq[n-1]$,
so that $\alpha\mapsto S(\alpha)$ is a bijection from the compositions
of $n$ onto the subsets of $[n-1]$.  Given $\alpha,\beta\vDash n$, we
say $\alpha$ \emph{refines} $\beta$ (equivalently, $\beta$
\emph{coarsens} $\alpha$) if $S(\beta)\subseteq S(\alpha)$;
concretely, $\beta$ is obtained from $\alpha$ by replacing consecutive
runs of parts by their sums.  Refinement is a partial order on the
compositions of $n$.

\begin{example}\label{ex:comp}
  The composition $\alpha=(2,1,2)\vDash 5$ has length $\ell=3$, partial
  sums $(s_0,s_1,s_2,s_3)=(0,2,3,5)$, and
  $\piu(\alpha)=s_1s_2s_3=2\cdot3\cdot5=30$.  Its subset is
  $S(\alpha)=\{s_1,s_2\}=\{2,3\}\subseteq[4]$, and it sorts to the
  partition $\lambda(\alpha)=(2,2,1)$.  There are three distinct
  rearrangements of $(2,2,1)$, namely $(2,2,1),(2,1,2),(1,2,2)$, with
  $\piu$-values $40,30,15$; observe that
  \[
  \frac1{40}+\frac1{30}+\frac1{15}
  =\frac{3+4+8}{120}=\frac18
  =\frac{1}{z_{(2,2,1)}},
  \qquad z_{(2,2,1)}=2^2\cdot2!\cdot1^1\cdot1!=8.
  \]
  This is no accident: Lemma~\ref{lem:zlambda} shows that the
  reciprocals $1/\piu(\alpha)$, summed over all rearrangements $\alpha$
  of any partition $\lambda$, always total $1/z_\lambda$.
\end{example}

\subsection{Alternating permutations and Euler numbers}\label{subsec:alt}
Let $B$ be a finite set of positive integers with $\#B=m$.  A
\emph{word on $B$} is a sequence $u=u_1u_2\dots u_m$ listing the
elements of $B$, each exactly once.  The word $u$ is
\emph{alternating} if
\[
u_1>u_2<u_3>u_4<\dots,
\]
that is, $u_i>u_{i+1}$ for odd $i$ and $u_i<u_{i+1}$ for even $i$,
$1\le i\le m-1$; it is \emph{reverse alternating} if all these
inequalities are reversed.  (This is the convention of~\cite{ASS};
alternating permutations should not be confused with even
permutations.)  Since the definitions depend only on the relative
order of the letters, the number of alternating words on $B$ depends
only on $m$; it is the \emph{Euler number} $E_m$, with $E_0=1$.  By a
classical theorem of Andr\'e~\cite{Andre} (see also the
survey~\cite{StanleySurvey}),
\begin{equation}\label{eq:andre}
  \sec t+\tan t\;=\;\sum_{k\ge0}E_k\,\frac{t^k}{k!},
\end{equation}
so $(E_0,E_1,E_2,\dots)=(1,1,1,2,5,16,61,272,1385,7936,50521,\dots)$.
Complementation of $B$ (replacing the $i$-th smallest letter of $B$ by
the $i$-th largest) reverses all inequalities and is an involution on
words on $B$; hence the number of reverse alternating words on $B$ is
also $E_m$.  We write $\Alt_{2n}$ for the set of alternating
permutations of $[2n]$, so $\#\Alt_{2n}=E_{2n}$.

The following counting lemma is used in~\cite[proof of
Theorem~3.2]{ASS}; we include the short proof.

\begin{lemma}\label{lem:blockwords}
  Let $B$ be a set of positive integers with $\#B=2m$, $m\ge1$.  The
  number of alternating words on $B$ whose first letter is $\max B$
  equals $E_{2m-1}$.
\end{lemma}

\begin{proof}
  Let $u=u_1\dots u_{2m}$ be a word on $B$ with $u_1=\max B$.  The
  condition $u_1>u_2$ holds automatically, so $u$ is alternating if and
  only if $u_2<u_3>u_4<\dotsb$, i.e., if and only if the word
  $u_2u_3\dots u_{2m}$ on the $(2m-1)$-element set $B\setminus\{\max
  B\}$ is reverse alternating.  By the complementation remark above,
  the number of such words is $E_{2m-1}$.
\end{proof}

\subsection{Records and record compositions}\label{subsec:records}
We recall the definitions from the Introduction.
Let $w=a_1a_2\dots a_{2n}\in\Alt_{2n}$.  Following~\cite[\S3]{ASS},
set
\[
\hat w\;\coloneqq \;a_1a_3a_5\dots a_{2n-1}\;=\;b_1b_2\dots b_n,
\qquad b_i\coloneqq a_{2i-1}.
\]
An index $i\in[n]$ is a \emph{record position} of $\hat w$ if
$b_i>b_j$ for all $1\le j<i$, that is, if $b_i$ is a left-to-right
maximum of $\hat w$; in particular $1$ is always a record
position.  Let
$\rec(\hat w)=\{r_1<r_2<\dots<r_\ell\}$ (so $r_1=1$) be the set of
record positions and put $r_{\ell+1}\coloneqq n+1$.  The \emph{record
composition} of $\hat w$ is
\[
\rc(\hat w)\;\coloneqq \;(r_2-r_1,\;r_3-r_2,\;\dots,\;r_{\ell+1}-r_\ell)
\;\vDash\;n,
\]
and the \emph{record partition} is
$\rp(\hat w)\coloneqq \lambda(\rc(\hat w))$, as in~\cite{ASS}.
Equivalently, $\rc(\hat w)$ lists the lengths of the factors obtained
by cutting $\hat w$ immediately before each record position other
than the first.  Note that
$\rc(\hat w)$ determines $\rec(\hat w)$ and conversely: if
$\rc(\hat w)=\alpha$ then $r_j=s_{j-1}(\alpha)+1$ for $1\le
j\le\ell(\alpha)$.

We will use the following elementary observation about
running maxima.  Call $i$ a \emph{record position} of a sequence
$b_1,\dots,b_n$ of distinct numbers if $b_i>b_j$ for all $j<i$.  Then
for every $i\in[n]$,
\begin{equation}\label{eq:runmax}
  \max\{b_1,\dots,b_i\}\;=\;b_r,\qquad
  r=\max\{\,r'\le i : r'\text{ is a record position}\,\}
\end{equation}
(the set on the right is nonempty, as $1$ is always a record
position); this is immediate by induction on $i$, since the running
maximum changes exactly when a new record occurs.

\subsection{Symmetric functions}\label{subsec:sym}
Throughout,  $K$ denotes a fixed field of
characteristic zero (the reader may take $K=\Q$ or $K=\mathbb{R}$).
We follow~\cite[Ch.~7]{EC2} for symmetric function notation. Namely,
$\Sym_K$ is the $K$-algebra of symmetric functions in
$x=(x_1,x_2,\dots)$, with monomial basis $m_\lambda$, complete
homogeneous basis $h_\lambda$, and power sum basis $p_\lambda$.

Sprout sequences are understood to be over $K$, as in~\cite{ASS}.
Whenever $\{u_\gamma\}$ is a basis of a vector space (or the monomials
of a formal power series ring), we write $[u_\gamma]G$ for the
coefficient of $u_\gamma$ in the expansion of an element $G$.  Thus
$[p_\lambda]f$ denotes the coefficient of $p_\lambda$ in a symmetric
function $f$, $[t^n]F(t)$ the coefficient of $t^n$ in a power series
$F(t)$, and, in Section~\ref{sec:nsym}, $[\Psi^\alpha]F$ the
coefficient of $\Psi^\alpha$, once $\{\Psi^\alpha\}$ has been shown to
be a basis.
For a
sprout sequence $\mathfrak{R}=(R_0,R_1,\dots)$ with seed $F$,
write, as in~\cite[eq.\ (1.2)]{ASS},
\begin{equation}\label{eq:seedb}
  \log F(t)=\sum_{n\ge1}b_n\frac{t^n}{n}.
\end{equation}
Then \cite[Theorem~2.1(d)]{ASS} gives the power sum expansion
\begin{equation}\label{eq:pexp}
  R_n\;=\;\sum_{\lambda\vdash n} z_\lambda^{-1}\,b_\lambda\,p_\lambda,
  \qquad b_\lambda\coloneqq b_{\lambda_1}b_{\lambda_2}\dots.
\end{equation}
For the seed $F(t)=\sec(\sqrt t\,)$, take first the odd part of
Andr\'e's theorem~\eqref{eq:andre}. Since $\sec$ is even and $\tan$ is
odd, $\tan x=\sum_{m\ge1}E_{2m-1}\,x^{2m-1}/(2m-1)!$\,.  Hence, we have
\[
\frac{d}{dx}\log\sec x=\tan x=\sum_{m\ge1}E_{2m-1}\frac{x^{2m-1}}{(2m-1)!},
\]
and so we have
\[
\log\sec x=\sum_{m\ge1}E_{2m-1}\frac{x^{2m}}{(2m)!},
\]
integrating termwise (both sides vanish at $x=0$).  Substituting
$x=\sqrt t$ gives
\[ \log\sec(\sqrt t\,)=\sum_{m\ge1}\bigl(mE_{2m-1}/(2m)!\bigr)\,t^m/m. \]
Namely, the coefficients \eqref{eq:seedb} for this seed are the
numbers $b_k=kE_{2k-1}/(2k)!$ of~\eqref{eq:bdef}.  This computation is
\cite[proof of Theorem~4.1]{ASS}, where it is shown that
$(A_0,A_1,\dots)$, with
$(2n)!A_n=\sum_{\lambda\vdash n}|\varphi(\lambda)|\,p_\lambda$, is the
sprout sequence with seed $\sec(\sqrt t\,)$. Equivalently, with these $b_k$,
\[ A_n=\sum_\lambda z_\lambda^{-1}b_\lambda p_\lambda. \]

\section{Counting alternating permutations by record composition}
\label{sec:count}

The strategy of this section is to prove Theorem~\ref{thm:main} by a
bijection.  Rather than count the permutations
$w\in\Alt_{2n}$ with $\rc(\hat w)=\alpha$ directly, we describe an
auxiliary set $\mathcal{B}(\alpha)$ of ``assembly instructions'', an
ordered list of blocks together with an alternating word on each
block, and show (see Lemmas~\ref{lem:concat} and~\ref{lem:decomp}) that
concatenating the words is a bijection onto the permutations we want,
under which the record composition reads off the block sizes in
order.  Counting $\mathcal{B}(\alpha)$ then splits into two
independent tasks: choosing the blocks (see Lemma~\ref{lem:orderedpartitions})
and decorating each with a word (see Lemma~\ref{lem:blockwords}).  The
product of the two counts is the formula of Theorem~\ref{thm:main}.
Throughout the section we fix $n\ge1$ and
$\alpha=(\alpha_1,\dots,\alpha_\ell)\vDash n$, with partial sums
$s_j=s_j(\alpha)$.

The following definition packages the assembly instructions.  Condition
(a) records \emph{which} letters go into each block and insists the
block maxima increase---this canonical ordering is what will let us
recover the blocks from $w$ alone, while condition (b) is exactly the
decoration counted by Lemma~\ref{lem:blockwords}.

\begin{definition}\label{def:blockdata}
  Let $\mathcal{B}(\alpha)$ denote the set of pairs
  $\bigl((B_1,\dots,B_\ell),(u^1,\dots,u^\ell)\bigr)$ where
  \begin{enumerate}
    \item[(a)] $(B_1,\dots,B_\ell)$ is an ordered set partition of $[2n]$ with
    $\#B_j=2\alpha_j$ and
    \[
    \max B_1<\max B_2<\dots<\max B_\ell,
    \]
    \item[(b)] for each $j$, $u^j$ is an alternating word on $B_j$ whose
    first letter is $\max B_j$.
  \end{enumerate}
  We call elements of $\mathcal{B}(\alpha)$ \emph{block data of type
  $\alpha$}, and define the \emph{concatenation map}
  \[
  \kappa\colon\mathcal{B}(\alpha)\longrightarrow
  \{\text{words on }[2n]\},\qquad
  \kappa\bigl((B_j),(u^j)\bigr)\coloneqq u^1u^2\dots u^\ell.
  \]
\end{definition}

Note that if $\bigl((B_j),(u^j)\bigr)\in\mathcal{B}(\alpha)$, then the
letters of the block $B_j$ occupy the positions
$2s_{j-1}+1,\,2s_{j-1}+2,\,\dots,\,2s_j$ of the concatenation
$w=\kappa\bigl((B_j),(u^j)\bigr)$; in particular each block begins at
an odd position and ends at an even position.

\begin{example}\label{ex:blockdata}
  Let $\alpha=(2,1,2)\vDash 5$, so we seek block data on $[10]$ with
  block sizes $2\alpha_j=4,2,4$.  One valid element of
  $\mathcal{B}(2,1,2)$ is
  \[
  (B_1,B_2,B_3)=\bigl(\{2,4,5,7\},\{3,8\},\{1,6,9,10\}\bigr),
  \qquad
  (u^1,u^2,u^3)=(7\,2\,5\,4,\;8\,3,\;10\,6\,9\,1).
  \]
  Indeed, the block maxima $7<8<10$ increase; each $u^j$ is an
  alternating word beginning with its block's maximum
  ($7\,2\,5\,4$ has $7>2<5>4$, and so on).  Concatenating gives
  $\kappa=7\,2\,5\,4\,8\,3\,10\,6\,9\,1$, which is exactly the
  permutation $w$ of the example in the Introduction, with $\rc(\hat w)=(2,1,2)$.
  Lemma~\ref{lem:concat} shows this always happens, and
  Lemma~\ref{lem:decomp} shows every such $w$ arises from a unique piece
  of block data.
\end{example}

\begin{lemma}\label{lem:concat}
  For every $\bigl((B_j),(u^j)\bigr)\in\mathcal{B}(\alpha)$, the word
  $w\coloneqq \kappa\bigl((B_j),(u^j)\bigr)$ is an alternating permutation of
  $[2n]$ and $\rc(\hat w)=\alpha$.
\end{lemma}

\begin{proof}
  That $w$ is a permutation of $[2n]$ is clear.  Write
  $p_j\coloneqq \max B_j$, so $p_1<\dots<p_\ell$.
  \smallskip

  \noindent
  \emph{$w$ is alternating}  Let $1\le p\le 2n-1$; we must show
  $w_p>w_{p+1}$ if $p$ is odd and $w_p<w_{p+1}$ if $p$ is even.  If
  positions $p$ and $p+1$ lie in the same block $B_j$, put
  $q\coloneqq p-2s_{j-1}$, so $1\le q\le 2\alpha_j-1$ and
  $(w_p,w_{p+1})=(u^j_q,u^j_{q+1})$.  Since $2s_{j-1}$ is even, $q$ and
  $p$ have the same parity, and the required inequality is exactly the
  alternating condition for $u^j$ at index $q$.  Otherwise $p$ is the
  last position of some block, so $p=2s_j$ with $1\le j\le \ell-1$ (by
  the note following Definition~\ref{def:blockdata}, and since
  $p<2n=2s_\ell$); this $p$ is even, and indeed
  $w_p\in B_j$ gives $w_p\le p_j<p_{j+1}=w_{p+1}$, using that
  $u^{j+1}$ begins with $p_{j+1}$.

  \smallskip
  \noindent
  \emph{Record positions.}  The odd positions of $w$ inside block
  $B_j$ are $2s_{j-1}+1,2s_{j-1}+3,\dots,2s_j-1$; hence the letters
  $\hat w_i$ for $s_{j-1}+1\le i\le s_j$ are precisely the letters of
  $u^j$ in odd positions, and in particular
  $\hat w_{s_{j-1}+1}=u^j_1=p_j$.  We claim
  $\rec(\hat w)=\{s_0+1,s_1+1,\dots,s_{\ell-1}+1\}$.  First,
  $s_{j-1}+1$ is a record position for every $j$: this is trivial for
  $j=1$, and for $j\ge2$ every letter of $\hat w$ strictly preceding
  position $s_{j-1}+1$ lies in $B_1\cup\dots\cup B_{j-1}$, hence is at
  most $\max(p_1,\dots,p_{j-1})=p_{j-1}<p_j=\hat w_{s_{j-1}+1}$.
  Second, if $s_{j-1}+1<i\le s_j$, then
  $\hat w_i\in B_j\setminus\{p_j\}$, so
  $\hat w_i<p_j=\hat w_{s_{j-1}+1}$ and $i$ is not a record position.
  Consequently, we have that
  $\rc(\hat w)=(s_1-s_0,\,s_2-s_1,\,\dots,\,s_\ell-s_{\ell-1})=\alpha$.
\end{proof}

\begin{lemma}\label{lem:decomp}
  The map $\kappa$ is a bijection from $\mathcal{B}(\alpha)$ onto
  $\{w\in\Alt_{2n}:\rc(\hat w)=\alpha\}$.
\end{lemma}

\begin{proof}
  By Lemma~\ref{lem:concat}, $\kappa$ maps $\mathcal{B}(\alpha)$ into
  the displayed set.  We construct the inverse.  Let $w\in\Alt_{2n}$
  with $\rc(\hat w)=\alpha$; as noted in
  Section~\ref{subsec:records}, its record positions are then
  $r_j=s_{j-1}+1$ for $1\le j\le\ell$, and we again put
  $r_{\ell+1}=n+1$.  Define, for $1\le j\le \ell$,
  \[
  u^j\;\coloneqq \;w_{2r_j-1}\,w_{2r_j}\dots w_{2r_{j+1}-2},
  \qquad B_j\;\coloneqq \;\{\text{letters of }u^j\},
  \]
  so $\#B_j=2(r_{j+1}-r_j)=2\alpha_j$; since the factors
  $u^1,\dots,u^\ell$ occupy the consecutive position ranges
  $[2r_j-1,\,2r_{j+1}-2]$, whose union is $[1,2n]$, the sets
  $(B_1,\dots,B_\ell)$ form an ordered set partition of $[2n]$.  We
  check that
  $\bigl((B_j),(u^j)\bigr)\in\mathcal{B}(\alpha)$.

  \smallskip
  \noindent
  \emph{Each $u^j$ is alternating.}  The $q$-th letter of $u^j$ is
  $w_p$ with $p=2r_j-2+q$; since $2r_j-2$ is even, $p$ and $q$ have the
  same parity, and the alternating inequalities for $u^j$ are inherited
  from those of $w$.

  \smallskip
  \noindent
  \emph{The first letter of $u^j$ is $\max B_j$, and the maxima
  increase.}  The first letter of $u^j$ is $w_{2r_j-1}=\hat w_{r_j}$,
  the $j$-th record value; write $p_j\coloneqq \hat w_{r_j}$, so
  $p_1<p_2<\dots<p_\ell$ since record values strictly increase.  The
  letters of $B_j$ at odd positions of $w$ are
  $\hat w_{r_j},\hat w_{r_j+1},\dots,\hat w_{r_{j+1}-1}$; for
  $r_j<i<r_{j+1}$ the index $i$ is not a record position, so by
  \eqref{eq:runmax} (the largest record position $\le i-1$ being $r_j$)
  we get $\hat w_i<\max\{\hat w_1,\dots,\hat w_{i-1}\}=\hat w_{r_j}=p_j$.
  The letters of $B_j$ at even positions of $w$ are $w_{2i}$ for
  $r_j\le i\le r_{j+1}-1$, and since $w$ is alternating,
  $w_{2i}<w_{2i-1}=\hat w_i\le p_j$.  Hence $\max B_j=p_j$, attained at
  the first letter of $u^j$, and $\max B_1<\dots<\max B_\ell$.

  Thus, $\bigl((B_j),(u^j)\bigr)\in\mathcal{B}(\alpha)$, and clearly
  $\kappa\bigl((B_j),(u^j)\bigr)=w$; so $\kappa$ is surjective onto the
  displayed set.  For injectivity, note that the proof of
  Lemma~\ref{lem:concat} shows that for any block data of type $\alpha$,
  the record positions of $\widehat{\kappa(\cdot)}$ occur exactly at the
  block starts $s_{j-1}+1$; hence the cut points
  $2r_j-1=2s_{j-1}+1$, and with them the words $u^j$ and sets $B_j$, are
  recovered from $w=\kappa\bigl((B_j),(u^j)\bigr)$ by the construction
  above.  The two maps are therefore mutually inverse.
\end{proof}

It remains to count $\mathcal{B}(\alpha)$.  We state the counting
lemma for ordered set partitions in the generality needed again in
Section~\ref{sec:general}.

\begin{lemma}\label{lem:orderedpartitions}
  Let $X$ be a finite set of positive integers, let $\ell\ge1$, and let
  $m_1,\dots,m_\ell\ge1$ satisfy $m_1+\dots+m_\ell=\#X$.  The number of
  ordered set partitions $(B_1,\dots,B_\ell)$ of $X$ with $\#B_j=m_j$
  for all $j$ and $\max B_1<\max B_2<\dots<\max B_\ell$ equals
  \[
  \prod_{j=1}^{\ell}\binom{m_1+\dots+m_j-1}{m_j-1}.
  \]
\end{lemma}

\begin{proof}
  The count depends only on $\#X$, since replacing $X$ by an
  order-isomorphic set induces a bijection on such ordered partitions.
  We induct on $\ell$.  If $\ell=1$ the count is $1=\binom{m_1-1}{m_1-1}$,
  as $B_1=X$ is forced.  For $\ell\ge2$: the element $\max X$ lies in
  some block $B_j$, and then $\max B_j=\max X\ge\max B_\ell$; since the
  block maxima strictly increase, $j=\ell$.  Thus $B_\ell$ consists of
  $\max X$ together with any $(m_\ell-1)$-element subset of
  $X\setminus\{\max X\}$, chosen in $\binom{\#X-1}{m_\ell-1}$ ways, and
  $(B_1,\dots,B_{\ell-1})$ is an arbitrary ordered set partition of the
  $(m_1+\dots+m_{\ell-1})$-element set $X\setminus B_\ell$ with block
  sizes $m_1,\dots,m_{\ell-1}$ and increasing maxima; the constraint
  $\max B_{\ell-1}<\max B_\ell=\max X$ is automatic.  The claim follows
  by induction, since $\#X-1 = m_1+\dots+m_\ell-1$.
\end{proof}

\begin{example}\label{ex:opcount}
  For $\alpha=(2,1,2)$ (block sizes $m_j=2\alpha_j=4,2,4$ on $X=[10]$),
  Lemma~\ref{lem:orderedpartitions} gives
  $\binom{3}{3}\binom{5}{1}\binom{9}{3}=1\cdot5\cdot84=420$ ordered set
  partitions with increasing maxima.  Decorating each block by an
  alternating word beginning with its maximum contributes, by
  Lemma~\ref{lem:blockwords}, a further
  $E_3\cdot E_1\cdot E_3=2\cdot1\cdot2=4$ choices, for a total of
  $420\cdot4=1680=N(2,1,2)$, matching the direct evaluation after
  Theorem~\ref{thm:main}.
\end{example}

\begin{proof}[Proof of Theorem~\ref{thm:main}]
  By Lemma~\ref{lem:decomp},
  $N(\alpha)=\#\mathcal{B}(\alpha)$.  Block data of type $\alpha$ are
  built by first choosing the ordered set partition
  $(B_1,\dots,B_\ell)$ of $[2n]$ with $\#B_j=2\alpha_j$ and increasing
  maxima, in
  $\prod_{j}\binom{2\alpha_1+\dots+2\alpha_j-1}{2\alpha_j-1}
  =\prod_j\binom{2s_j-1}{2\alpha_j-1}$
  ways (Lemma~\ref{lem:orderedpartitions} with $m_j=2\alpha_j$), and
  then choosing, independently for each $j$, an alternating word on
  $B_j$ beginning with $\max B_j$, in $E_{2\alpha_j-1}$ ways
  (Lemma~\ref{lem:blockwords}).
\end{proof}

\begin{remark}\label{rem:positive}
  Every factor in Theorem~\ref{thm:main} is a positive integer, so
  $N(\alpha)>0$ for every $\alpha\vDash n$: the record composition
  statistic is surjective onto compositions of $n$.
\end{remark}

\section{A product formula and the record partition theorem}\label{sec:product}
\subsection{Deducing Theorem~\ref{thm:bform} from Theorem~\ref{thm:main}}
Theorem~\ref{thm:main} expresses $N(\alpha)$ through binomial
coefficients and Euler numbers.  In this section we recast it in the
multiplicative shape $N(\alpha)=(2n)!\,b_\alpha/\piu(\alpha)$
(Theorem~\ref{thm:bform}), where the $b_k$ are the seed coefficients
of $\sec(\sqrt t\,)$.  This is the form that will match the
noncommutative expansion in Section~\ref{sec:nsym}, and it also lets
us recover the original partition-indexed count of~\cite{ASS} by
summing over rearrangements (Corollary~\ref{cor:recover}).  The bridge
between the two shapes is the following elementary identity, used here
with $d=2$ and again with $d=1$ in Section~\ref{sec:general}.

\begin{lemma}\label{lem:telescope} If $d\ge1$ and $\alpha=(\alpha_1,\dots,\alpha_\ell)\vDash n$ with
  partial sums $s_j$, then we have
  \[
  \prod_{j=1}^{\ell}\binom{ds_j-1}{d\alpha_j-1}
  \;=\;(dn)!\;\prod_{j=1}^{\ell}\frac{1}{(ds_j)\,(d\alpha_j-1)!}\,.
  \]
\end{lemma}

\begin{proof}
  Since $ds_j-d\alpha_j=ds_{j-1}$, each factor equals
  $\binom{ds_j-1}{d\alpha_j-1}
  =\dfrac{(ds_j-1)!}{(d\alpha_j-1)!\,(ds_{j-1})!}$. Splitting and using
  $(ds_j-1)!=(ds_j)!/(ds_j)$ we get
  \[
  \prod_{j=1}^{\ell}\frac{(ds_j-1)!}{(ds_{j-1})!}
  =\prod_{j=1}^{\ell}\frac{(ds_j)!}{(ds_j)\,(ds_{j-1})!}
  =\frac{(ds_\ell)!}{(ds_0)!}\prod_{j=1}^{\ell}\frac{1}{ds_j}
  =(dn)!\prod_{j=1}^{\ell}\frac{1}{ds_j},
  \]
  the middle equality by telescoping ($s_0=0$, $s_\ell=n$, $0!=1$).
  Multiplying by $\prod_j 1/(d\alpha_j-1)!$ gives the claim.
\end{proof}

\begin{proof}[Proof of Theorem~\ref{thm:bform}]
  By Theorem~\ref{thm:main} and Lemma~\ref{lem:telescope} with $d=2$,
  \[
  N(\alpha)
  =(2n)!\prod_{j=1}^{\ell}\frac{E_{2\alpha_j-1}}{(2s_j)\,(2\alpha_j-1)!}.
  \]
  On the other hand, by~\eqref{eq:bdef} and
  $(2\alpha_j)!=(2\alpha_j)(2\alpha_j-1)!$,
  \[
  \frac{b_{\alpha_j}}{s_j}
  =\frac{\alpha_j\,E_{2\alpha_j-1}}{(2\alpha_j)!\;s_j}
  =\frac{E_{2\alpha_j-1}}{2\,(2\alpha_j-1)!\;s_j}
  =\frac{E_{2\alpha_j-1}}{(2s_j)\,(2\alpha_j-1)!}\,,
  \]
  so the two products agree factor by factor.  Finally
  $\prod_j s_j=\piu(\alpha)$ by definition.
\end{proof}

\subsection{Deducing \eqref{eq:ass32} from Theorem~\ref{thm:main}}
We now show that summing over rearrangements recovers
\eqref{eq:ass32}.  The key is the following identity.

\begin{lemma}\label{lem:zlambda}
  For every partition $\lambda\vdash n$, we have
  \[
  \sum_{\alpha\sim\lambda}\frac{1}{\piu(\alpha)}
  \;=\;\frac{1}{z_\lambda},
  \]
  the sum being over all distinct rearrangements $\alpha$ of $\lambda$.
\end{lemma}

\begin{proof}
  Induct on $\ell(\lambda)$, the case $\ell=0$ (empty partition,
  $z_\varnothing=1$) being trivial.  Let $\ell\ge1$ and write
  $f(\lambda)$ for the left-hand side.  Group the rearrangements
  $\alpha$ of $\lambda$ according to the value $k$ of the last part
  $\alpha_\ell$: the possible values are the distinct parts of
  $\lambda$, and deleting the last part gives a bijection between
  rearrangements of $\lambda$ with last part $k$ and rearrangements of
  the partition $\lambda\setminus k$ obtained from $\lambda$ by removing
  one part equal to $k$.  Since $s_\ell(\alpha)=n$ for all
  $\alpha\sim\lambda$, and $\piu(\alpha)=n\cdot\piu(\alpha')$ where
  $\alpha'$ is $\alpha$ with the last part deleted, we have
  \[
  f(\lambda)=\frac1n\sum_{k}f(\lambda\setminus k)
  =\frac1n\sum_{k}\frac{1}{z_{\lambda\setminus k}}
  =\frac1n\sum_{k}\frac{k\,m_k(\lambda)}{z_\lambda}
  =\frac{1}{z_\lambda}\cdot\frac1n\sum_{k}k\,m_k(\lambda)
  =\frac{1}{z_\lambda},
  \]
  where $k$ runs over the distinct parts of $\lambda$, the second
  equality is the inductive hypothesis, the third is the computation
  $z_{\lambda\setminus k}=z_\lambda/(k\,m_k(\lambda))$ (removing one
  part equal to $k$ replaces the factor $k^{m_k}m_k!$ of $z_\lambda$ by
  $k^{m_k-1}(m_k-1)!$), and
  the last uses $\sum_k k\,m_k(\lambda)=n$.
\end{proof}

\begin{remark}
  Lemma~\ref{lem:zlambda} is classical.  For instance, it also follows
  by comparing the coefficient of $p_\lambda$ in the classical expansion
  $h_n=\sum_{\mu\vdash n}z_\mu^{-1}p_\mu$ (obtained by exponentiating
  the identity $H'(t)/H(t)=\sum_{k\ge1}p_kt^{k-1}$ derived in the proof
  of Lemma~\ref{lem:chi} below) with the image, under the projection
  $\chi$ of Section~\ref{sec:nsym}, of the expansion of $S_n$ in
  Proposition~\ref{prop:gkllrt} below.  We included the direct proof to
  keep the treatment self-contained.
\end{remark}

\begin{corollary}\label{cor:recover}
  Theorem~\ref{thm:main} implies \eqref{eq:ass32}. Namely, for every
  $\lambda\vdash n$, we have
  \[
  \sum_{\alpha\sim\lambda}N(\alpha)\;=\;(2n)!\,\frac{b_\lambda}{z_\lambda}
  \;=\;|\varphi(\lambda)|
  \;=\;\#\{w\in\Alt_{2n}:\rp(\hat w)=\lambda\}.
  \]
\end{corollary}

\begin{proof}
  Since $\rp(\hat w)=\lambda$ if and only if $\rc(\hat w)\sim\lambda$,
  the last equality follows by summing Theorem~\ref{thm:main} over
  $\alpha\sim\lambda$; it remains to evaluate the sum.  The product
  $b_\alpha\coloneqq b_{\alpha_1}\dots b_{\alpha_\ell}$ equals $b_\lambda$ for
  every $\alpha\sim\lambda$, so by Theorem~\ref{thm:bform} and
  Lemma~\ref{lem:zlambda},
  \[
  \sum_{\alpha\sim\lambda}N(\alpha)
  =(2n)!\,b_\lambda\sum_{\alpha\sim\lambda}\frac{1}{\piu(\alpha)}
  =(2n)!\,\frac{b_\lambda}{z_\lambda}.
  \]
  Finally, with $b_k=kE_{2k-1}/(2k)!$ and $m_k=m_k(\lambda)$, we have
  \[
  \frac{b_\lambda}{z_\lambda}
  =\prod_{k\ge1}\frac{b_k^{\,m_k}}{k^{m_k}\,m_k!}
  =\prod_{k\ge1}\frac{1}{m_k!}
  \left(\frac{k\,E_{2k-1}}{(2k)!\,k}\right)^{m_k}
  =\prod_{k\ge1}\frac{1}{m_k!}
  \left(\frac{E_{2k-1}}{(2k)!}\right)^{m_k},
  \]
  which is $|\varphi(\lambda)|/(2n)!$ by~\eqref{eq:phiintro}.
\end{proof}

We emphasize that Corollary~\ref{cor:recover} does not constitute an
essentially new proof of \eqref{eq:ass32}: as explained in the
introduction, our bijection is the construction of \cite[proof of
Theorem~3.2]{ASS} with the canonical ordering of the blocks retained.

\section{Noncommutative symmetric functions}\label{sec:nsym}

So far the record composition has been a purely combinatorial
statistic.  We now explain the algebraic phenomenon behind it,
answering the second half of the problem of~\cite{ASS}.  The guiding
idea is simple to state.  The symmetric function $A_n$ lives in the
commutative algebra $\Sym_K$, where its power sum expansion is indexed
by \emph{partitions} and therefore cannot distinguish orderings.
There is, however, a noncommutative algebra $\NSym$ sitting above
$\Sym_K$, its elements are indexed by \emph{compositions}, together
with a projection $\chi\colon\NSym\to\Sym_K$ that forgets the ordering.
We will produce a canonical element $\bA_n\in\NSym$ with
$\chi(\bA_n)=A_n$ whose coordinates, in the right basis, are exactly
the record composition counts $N(\alpha)$.  In other words, $N(\alpha)$
is what one sees before $\chi$ erases the order.

Section~\ref{subsec:nsymbg} collects the facts we need about $\NSym$,
with proofs included so the paper is self-contained; readers new to
the subject lose nothing by treating $\NSym$ concretely as the algebra
of polynomials in noncommuting variables $S_1,S_2,\dots$.
Section~\ref{subsec:lift} constructs the lift, proves
Theorem~\ref{thm:nsym}, and records the negative results that delimit
it.

\subsection{Background on $\NSym$}\label{subsec:nsymbg}
We recall what we need from~\cite{GKLLRT}, working over the field $K$
of Section~\ref{subsec:sym}.  Let $\NSym$ denote the free associative
$K$-algebra $K\langle S_1,S_2,S_3,\dots\rangle$ on noncommuting
generators $S_k$ of degree $k$, with $S_0\coloneqq 1$, and let
\[
\sigma(t)\;\coloneqq \;\sum_{k\ge0}S_k\,t^k\;\in\;\NSym[[t]],
\]
where $t$ is a central indeterminate.  In~\cite[\S3.1]{GKLLRT} the
algebra is instead presented as the free algebra on noncommutative
elementary functions $\Lambda_k$, and the $S_k$ are defined there by
\[ \sigma(t)\cdot\sum_{k\ge0}(-t)^k\Lambda_k=1. \]
That the $S_k$ so
defined form another family of free generators, so that the two
presentations yield isomorphic graded algebras, is established
in~\cite{GKLLRT}.  We make no use of the $\Lambda_k$: everything in
this section is developed from the $S$-presentation alone, which we
take as the definition.

The words
$S^\alpha\coloneqq S_{\alpha_1}S_{\alpha_2}\dots S_{\alpha_\ell}$, for
$\alpha\vDash n$, form a basis of the homogeneous component
$\NSym_n$; here and below, an empty product of generators is $1$, so
that $S^\varnothing=1$ spans $\NSym_0=K$.
Following \cite[Definition~3.1]{GKLLRT}, the \emph{noncommutative
power sums of the first kind} $\Psi_k$ are defined by
\begin{equation}\label{eq:ode}
  \frac{d}{dt}\,\sigma(t)\;=\;\sigma(t)\,\psi(t),
  \qquad \psi(t)\coloneqq \sum_{k\ge1}\Psi_k\,t^{k-1};
\end{equation}
comparing coefficients of $t^{n-1}$, this is equivalent to
\begin{equation}\label{eq:newton}
  n\,S_n\;=\;\sum_{k=1}^{n}S_{n-k}\,\Psi_k\qquad(n\ge1),
\end{equation}
which determines $\Psi_n=nS_n-\sum_{k=1}^{n-1}S_{n-k}\Psi_k$
recursively; each $\Psi_n$ is homogeneous of degree $n$ and, by
induction on \eqref{eq:newton}, has \emph{integer} coordinates in the
basis $\{S^\alpha\}$.  For $\alpha\vDash n$ put
$\Psi^\alpha\coloneqq \Psi_{\alpha_1}\dots\Psi_{\alpha_\ell}$, so
$\Psi^\varnothing=1$.  The
\emph{noncommutative power sums of the second kind} $\Phi_k$ are
defined by $\sigma(t)=\exp\bigl(\sum_{k\ge1}\Phi_kt^k/k\bigr)$
\cite[Definition~3.1]{GKLLRT}; they will appear only in
Remark~\ref{rem:phi}.

\begin{example}\label{ex:psi}
  Running the recursion for small $n$ gives
  \[
  \Psi_1=S_1,\qquad
  \Psi_2=2S_2-S_1^2,\qquad
  \Psi_3=3S_3-S_2S_1-2S_1S_2+S_1^3,
  \]
  each with integer $S$-word coordinates, as claimed.  As a check of
  Proposition~\ref{prop:gkllrt} at $n=2$: the compositions of $2$ are
  $(2)$ and $(1,1)$, with $\piu(2)=2$ and $\piu(1,1)=1\cdot2=2$, so
  \[
  \frac{\Psi^{(2)}}{\piu(2)}+\frac{\Psi^{(1,1)}}{\piu(1,1)}
  =\frac{2S_2-S_1^2}{2}+\frac{S_1^2}{2}=S_2,
  \]
  since $\Psi^{(1,1)}=\Psi_1\Psi_1=S_1^2$.
\end{example}

The following expansion is \cite[Proposition~4.5]{GKLLRT}
(specialized to a one-part composition); we include the short proof
because its induction is the exact algebraic shadow of the bijection
of Section~\ref{sec:count}.

\begin{proposition}[{\cite[Proposition~4.5]{GKLLRT}}]\label{prop:gkllrt}
  For all $n\ge1$, we have
  \[
  S_n\;=\;\sum_{\alpha\vDash n}\frac{\Psi^\alpha}{\piu(\alpha)}.
  \]
\end{proposition}

\begin{proof}
  We prove the identity for all $n\ge0$ and induct on $n$; the base
  case $n=0$ reads $S_0=1=\Psi^\varnothing/\piu(\varnothing)$, which
  holds by our conventions.  For $n\ge1$, by \eqref{eq:newton} and the
  inductive hypothesis applied to $S_{n-k}$ for $1\le k\le n$,
  \[
  S_n=\frac1n\sum_{k=1}^{n}S_{n-k}\Psi_k
  =\frac1n\sum_{k=1}^{n}\;\sum_{\beta\vDash n-k}
  \frac{\Psi^\beta}{\piu(\beta)}\,\Psi_k .
  \]
  Every $\alpha\vDash n$ arises exactly once in this double sum, as the
  concatenation of $\beta\vDash n-k$ with the one-part composition
  $(k)$, where $k=\alpha_{\ell(\alpha)}$ is the last part of $\alpha$
  (the case $k=n$, with $\beta$ empty, included); and
  $\piu(\alpha)=\piu(\beta)\cdot n$ since the new partial sum is
  $s_{\ell(\alpha)}(\alpha)=n$.
\end{proof}

\begin{corollary}\label{cor:psibasis}
  For every $\beta\vDash n$,
  $S^\beta=\sum_{\alpha}c_{\alpha,\beta}\,\Psi^\alpha$, where $\alpha$
  runs over the compositions of $n$ refining $\beta$ and
  $c_{\beta,\beta}=\prod_p 1/\beta_p\ne0$.  Consequently
  $\{\Psi^\alpha:\alpha\vDash n\}$ is a basis of $\NSym_n$, and
  $\Psi_1,\Psi_2,\dots$ freely generate $\NSym$.
\end{corollary}

\begin{proof}
  Write $m\coloneqq \ell(\beta)$.  Multiplying the expansions of
  Proposition~\ref{prop:gkllrt} for
  $S_{\beta_1},\dots,S_{\beta_m}$ expresses $S^\beta$ as a sum of
  $\Psi^\alpha$ over concatenations
  $\alpha=\gamma^{(1)}\dots\gamma^{(m)}$ with $\gamma^{(p)}\vDash
  \beta_p$.  These concatenations are exactly the refinements of
  $\beta$: every refinement $\alpha$ of $\beta$ decomposes uniquely as
  such a concatenation, by cutting $\alpha$ at the partial sums of
  $\beta$.  The coefficient of $\Psi^\beta$ itself (each
  $\gamma^{(p)}=(\beta_p)$) is
  $\prod_p\piu\bigl((\beta_p)\bigr)^{-1}=\prod_p 1/\beta_p$.  Thus the
  transition matrix from $\{S^\beta\}$ to $\{\Psi^\alpha\}$ is
  triangular with nonzero diagonal with respect to (any linear extension
  of) the refinement order on compositions of $n$, hence invertible;
  since $\{S^\beta\}$ is a basis of $\NSym_n$, so is
  $\{\Psi^\alpha\}$.  Finally, the algebra map from the free algebra on
  generators $y_1,y_2,\dots$ (with $\deg y_k=k$) to $\NSym$ sending
  $y_k\mapsto\Psi_k$ carries the word basis to the basis
  $\{\Psi^\alpha\}$, hence is an isomorphism onto $\NSym$.
\end{proof}

Let $\chi\colon\NSym\to\Sym_K$ be the algebra homomorphism with
$\chi(S_n)=h_n$ (well defined by freeness); this is the standard
commutative projection of \cite{GKLLRT}.

\begin{lemma}\label{lem:chi} For all $n\geq 1$, we have
  $\chi(\Psi_n)=p_n$. In particular, we have
  $\chi(\Psi^\alpha)=p_{\lambda(\alpha)}$.
\end{lemma}

\begin{proof}
  Let $H(t)\coloneqq \sum_{n\ge0}h_nt^n=\prod_{i\ge1}(1-x_it)^{-1}$.  Then
  $\log H(t)=-\sum_i\log(1-x_it)$, so
  $H'(t)/H(t)=\sum_i x_i/(1-x_it)=\sum_{n\ge1}p_nt^{n-1}$, i.e.,
  $nh_n=\sum_{k=1}^n h_{n-k}p_k$ for all $n\ge1$.  Applying $\chi$ to
  \eqref{eq:newton} gives $nh_n=\sum_{k=1}^n h_{n-k}\chi(\Psi_k)$.
  Since the system $nh_n=\sum_{k=1}^n h_{n-k}x_k$ ($n\ge1$) determines
  the $x_k$ uniquely (solve recursively:
  $x_n=nh_n-\sum_{k<n}h_{n-k}x_k$), we conclude $\chi(\Psi_n)=p_n$.
\end{proof}

Finally, we recall the ribbon basis, needed only for
Remark~\ref{rem:negative}, Question~\ref{q:ribbon}, and
Example~\ref{ex:negative}.  For $\alpha\vDash n$, we set
\begin{equation}\label{eq:ribbon}
  R_\alpha\;\coloneqq \;\sum_{\beta}(-1)^{\ell(\alpha)-\ell(\beta)}\,S^\beta,
\end{equation}
the sum over the coarsenings $\beta$ of $\alpha$; equivalently, by
M\"obius inversion over the Boolean lattice of subsets of $S(\alpha)$,
$S^\alpha=\sum_\beta R_\beta$, again summed over coarsenings.  Since
coarsening is a partial order and the coefficient of $S^\alpha$ in
\eqref{eq:ribbon} is $1$, the family $\{R_\alpha:\alpha\vDash n\}$ is
a basis of $\NSym_n$, with integer transition matrices to and from
$\{S^\alpha\}$.  This is one of the equivalent descriptions of the
\emph{ribbon basis} of \cite{GKLLRT}; its commutative image
$\chi(R_\alpha)$ is the skew Schur function of a ribbon
(border-strip) shape determined by $\alpha$, by the classical
alternating-sum formula for ribbon Schur functions in terms of the
$h_\mu$ (see \cite{GKLLRT} and \cite[Ch.~7]{EC2}).

\subsection{The canonical lift of a sprout sequence}
\label{subsec:lift}

We can now make precise the ``element sitting above $A_n$'' promised at
the start of the section.  The idea is to take the expansion
$S_n=\sum_\alpha\Psi^\alpha/\piu(\alpha)$ of
Proposition~\ref{prop:gkllrt} and rescale the generator $\Psi_k$ by
the seed coefficient $b_k$; this is exactly what a sprout sequence
does at the commutative level, and the rescaling is realized by an
algebra endomorphism of $\NSym$.

\begin{definition}\label{def:lift}
  Let $b=(b_1,b_2,\dots)$ be a sequence of elements of $K$.  By
  Corollary~\ref{cor:psibasis} there is a unique algebra
  endomorphism $\phi_b$ of $\NSym$ with $\phi_b(\Psi_k)=b_k\Psi_k$ for
  all $k$, and it is graded, each $b_k\Psi_k$ being homogeneous of
  degree $k$.  Given a sprout sequence $\mathfrak{R}=(R_0,R_1,\dots)$ over
  $K$ with seed $F$, let $b$ be as in \eqref{eq:seedb} and define the
  \emph{canonical noncommutative lift}
  \[
  \bR_n\;\coloneqq \;\phi_b(S_n)\;\in\;\NSym_n.
  \]
  For the seed $\sec(\sqrt t\,)$, so $b_k=kE_{2k-1}/(2k)!$ as
  in~\eqref{eq:bdef}, we write $\bA_n\coloneqq \phi_b(S_n)$.
\end{definition}

\begin{theorem}\label{thm:ncsprout}
  Let $\mathfrak{R}$ be a sprout sequence over $K$ with seed
  coefficients $b$ as in \eqref{eq:seedb}, and write
  $b_\alpha\coloneqq b_{\alpha_1}\dots b_{\alpha_\ell}$ for
  $\alpha\vDash n$.  Then we have
  \[
  \text{(i)}\quad
  \bR_n=\sum_{\alpha\vDash n}\frac{b_\alpha}{\piu(\alpha)}\,\Psi^\alpha,
  \qquad\qquad
  \text{(ii)}\quad \chi(\bR_n)=R_n .
  \]
\end{theorem}

\begin{proof}
  (i) Apply $\phi_b$ to Proposition~\ref{prop:gkllrt} and use
  $\phi_b(\Psi^\alpha)=b_\alpha\Psi^\alpha$.  (ii) By (i),
  Lemma~\ref{lem:chi}, and Lemma~\ref{lem:zlambda},
  \[
  \chi(\bR_n)
  =\sum_{\alpha\vDash n}\frac{b_\alpha}{\piu(\alpha)}\,p_{\lambda(\alpha)}
  =\sum_{\lambda\vdash n}b_\lambda\,p_\lambda
  \sum_{\alpha\sim\lambda}\frac{1}{\piu(\alpha)}
  =\sum_{\lambda\vdash n}z_\lambda^{-1}\,b_\lambda\,p_\lambda
  = R_n,
  \]
  the last equality by \eqref{eq:pexp}, i.e., by
  \cite[Theorem~2.1(d)]{ASS}.
\end{proof}

\begin{proof}[Proof of Theorem~\ref{thm:nsym}]
  For every $\alpha\vDash n$,
  \[
  (2n)!\,[\Psi^\alpha]\,\bA_n
  =(2n)!\,\frac{b_\alpha}{\piu(\alpha)}
  =N(\alpha)
  =\#\{w\in\Alt_{2n}:\rc(\hat w)=\alpha\},
  \]
  where the first equality is Theorem~\ref{thm:ncsprout}(i) applied to
  the seed $\sec(\sqrt t\,)$, the second is Theorem~\ref{thm:bform},
  and the third is the definition of $N(\alpha)$ in
  Theorem~\ref{thm:main}.  Theorem~\ref{thm:ncsprout}(ii) gives
  $\chi(\bA_n)=A_n$, and applying $\chi$ to the displayed expansion
  recovers, via Lemma~\ref{lem:chi} and
  Corollary~\ref{cor:recover}, the power sum expansion
  $(2n)!A_n=\sum_\lambda|\varphi(\lambda)|\,p_\lambda$.
\end{proof}

Concretely, for $n=2,3$ this reads
$4!\,\bA_2=3\,\Psi^{(1,1)}+2\,\Psi^{(2)}$ and
$6!\,\bA_3=15\,\Psi^{(1,1,1)}+20\,\Psi^{(1,2)}+10\,\Psi^{(2,1)}
+16\,\Psi^{(3)}$, the coefficients being the entries of
Table~\ref{tab:small}.  Example~\ref{ex:lifts} carries out these
computations and converts them to the $S$-word basis, and
Example~\ref{ex:negative} exhibits the first negative $S$-word and
ribbon coefficients at $n=4$.

\begin{remark}[Why the first kind]\label{rem:ode}
  The parallel between Proposition~\ref{prop:gkllrt} and
  Theorem~\ref{thm:main} is exact: the induction proving the former
  peels off the last factor $\Psi_k$ at the cost of the partial sum
  $s_\ell=n$, while the bijection proving the latter peels off the block
  containing the letter $2n$ (necessarily the last block) at the cost of
  the binomial factor $\binom{2n-1}{2k-1}$.  The commutative shadow of
  this common recursion is the classical identity
  \[
  E_{2n}\;=\;\sum_{k=1}^{n}\binom{2n-1}{2k-1}E_{2k-1}E_{2n-2k},
  \]
  obtained from Theorem~\ref{thm:main} by summing over all
  $\alpha\vDash n$ with last part $k$: for $k<n$, deleting the last
  part and using $\sum_{\beta\vDash n-k}N(\beta)=E_{2n-2k}$ (every
  element of $\Alt_{2n-2k}$ has exactly one record composition), and
  for $k=n$, observing that the single composition $\alpha=(n)$
  contributes $N(n)=E_{2n-1}=\binom{2n-1}{2n-1}E_{2n-1}E_0$.  In terms
  of exponential generating functions this is
  $\sec' = \sec\cdot\tan$, of which the defining relation
  \eqref{eq:ode}, $\sigma'=\sigma\psi$, is the noncommutative lift.  In
  this precise sense the record composition statistic ``lives in'' the
  basis $\{\Psi^\alpha\}$.
\end{remark}

\begin{remark}[The second kind does not fit]\label{rem:phi}
  From the formal identity
  \[ \sigma(t)=\exp\bigl(\sum_k\Phi_kt^k/k\bigr), \]
  one obtains, by
  expanding the exponential, the identity
  \[ S_n=\sum_{\alpha\vDash n}
  \Phi^\alpha/\bigl(\ell(\alpha)!\,\alpha_1\dots\alpha_\ell\bigr) \]
  in $\NSym$ (no basis property of the $\Phi^\alpha$ is needed here).
  Rescaling the summand indexed by $\alpha$ by $b_\alpha$, exactly as
  Theorem~\ref{thm:ncsprout}(i) rescales the $\Psi$-expansion, produces
  the numbers
  $(2n)!\,b_\alpha/\bigl(\ell(\alpha)!\prod_j\alpha_j\bigr)$, which are
  \emph{symmetric} in the parts of $\alpha$.  For $n=3$, they equal $15$
  for both orderings of the parts $\{1,2\}$, whereas $N(1,2)=20$ and
  $N(2,1)=10$.  Consequently, no statement of the form ``$(2n)!$ times
  the analogously rescaled $\Phi$-coefficient equals $N(\alpha)$'' can
  hold. The record composition statistic distinguishes orderings that
  the $\Phi$-coefficients cannot.  Only the averaged (commutative)
  information survives in the $\Phi$-picture.
\end{remark}

\begin{remark}[Integrality and failure of positivity in other bases]
  \label{rem:negative}
  By Theorem~\ref{thm:nsym}, $(2n)!\,\bA_n$ has nonnegative integer
  coordinates in the basis $\{\Psi^\alpha\}$.  Its coordinates in the
  basis $\{S^\alpha\}$, and in the ribbon basis $\{R_\alpha\}$ of
  \cite{GKLLRT}, are also integers. Indeed, we have
  \[ (2n)!\,\bA_n=\sum_\alpha N(\alpha)\Psi^\alpha, \]
  with
  $N(\alpha)\in\Z$. Each $\Psi^\alpha$ has integer $S$-word coordinates
  (noted after \eqref{eq:newton}), and the $S$-to-ribbon transition is
  integral in both directions (noted after \eqref{eq:ribbon}).
  However, these coordinates are \emph{not} nonnegative in
  general.  Exact computation in the free algebra (see
  Remark~\ref{rem:verification}) shows that nonnegativity holds in both
  bases for $n\le3$ and fails for every $4\le n\le 8$. For $n=4$, we have
  \[
  [S^{(1,2,1)}]\,8!\,\bA_4=-8,
  \qquad
  [R_{(1,2,1)}]\,8!\,\bA_4=-7,
  \]
  with all remaining coordinates positive (the full expansions are
  Example~\ref{ex:negative}).  Two consequences deserve emphasis.
  First, since $\chi(S^\alpha)=h_{\lambda(\alpha)}$, the $S$-word
  coordinates of $(2n)!\,\bA_n$ refine, as signed integers, the
  $h$-expansion coefficients of $(2n)!A_n$, which are nonnegative by
  \cite[Theorem~5.1]{ASS}; the computation above shows that the
  combinatorial interpretation of those $h$-coefficients, sought in the
  problem posed in \cite[\S5]{ASS}, cannot be obtained by refining along the
  $S$-word coordinates of the canonical lift $\bA_n$.  Second, the
  $h$-positivity of $(2n)!A_n$ does not lift to ribbon positivity of
  $\bA_n$ (nonnegativity of the coordinates in the basis
  $\{R_\alpha\}$), even though ribbon positivity does hold for some
  sprout lifts (e.g., for the seed $(1-t)^{-1}$ one has $b_k=1$ and
  $\bR_n=S_n=R_{(n)}$).
\end{remark}

\begin{question}\label{q:ribbon}
  For which sprout sequences over $\mathbb{R}$ is the canonical lift
  $\bR_n$ ribbon positive for all $n$?
\end{question}

This property is strictly stronger than the Schur positivity of the
sprout sequence itself.  Indeed, $\chi$ carries the ribbon basis to
the ribbon Schur functions, which are skew Schur functions and hence
Schur positive~\cite[Ch.~7]{EC2}.  Thus, by
Theorem~\ref{thm:ncsprout}(ii), if $\bR_n$ is ribbon positive then
$R_n$ is Schur positive.  The Schur positive sprout sequences are
classified in~\cite[Theorem~2.9]{ASS} through the Edrei--Thoma
theorem. Question~\ref{q:ribbon} asks which of them satisfy the
stronger ribbon condition.  The two are not equivalent.
Remark~\ref{rem:negative} exhibits a Schur positive sprout sequence
whose lift is \emph{not} ribbon positive.

\begin{remark}[The dual quasisymmetric picture]\label{rem:qsym}
  It is natural to ask whether the record composition refinement can be
  seen on the dual side, and the answer is that it cannot: the
  refinement is a property of the \emph{lift} $\bA_n$, not of $A_n$.  We
  explain this briefly.

  The graded dual of $\NSym$ is the algebra $\QSym$ of quasisymmetric
  functions.  Under the pairing, the projection $\chi\colon\NSym\to\Sym_K$
  is adjoint to the inclusion $\Sym_K\subseteq\QSym$, where the pairing
  on $\Sym_K$ is the Hall inner product $\langle\,,\rangle$ (so that
  $\langle p_\lambda,p_\mu\rangle=z_\lambda\delta_{\lambda\mu}$; see
  \cite[Ch.~7]{EC2}).  For this duality we refer to~\cite{GKLLRT}
  and~\cite{BDHMN}. The bases of $\QSym$ dual to $\{\Psi^\alpha\}$ and
  $\{\Phi^\alpha\}$ are, up to normalization, the type~1 and type~2
  quasisymmetric power sums of Ballantine, Daugherty, Hicks, Mason, and
  Niese~\cite{BDHMN}.

  Now let $f\in\Sym_K$ be symmetric of degree $n$.  By adjointness, we have
  \[
  \langle\Psi^\alpha,f\rangle
  =\langle p_{\lambda(\alpha)},f\rangle
  =z_{\lambda(\alpha)}\,[p_{\lambda(\alpha)}]f,
  \]
  which depends only on the underlying partition $\lambda(\alpha)$. For
  $f=A_n$, it equals $b_{\lambda(\alpha)}$.  Therefore, pairing $A_n$ against
  the basis dual to $\{\Psi^\alpha\}$ returns the same value for all
  compositions $\alpha$ with a given sort $\lambda(\alpha)$.  The dual
  pairing is therefore blind to the order of the parts, which is exactly
  the information the record composition records.  This is why the
  refinement is visible only after lifting $A_n$ to $\bA_n$.
\end{remark}

\section{A general refinement for exponential seeds}\label{sec:general}

The count in Theorem~\ref{thm:main} is built from two ingredients:
ordered blocks with increasing maxima
(see Lemma~\ref{lem:orderedpartitions}), and an independently chosen
structure on each block (see Lemma~\ref{lem:blockwords}).  Only the
second ingredient refers to alternating permutations, and only the
bijection of Lemmas~\ref{lem:concat} and~\ref{lem:decomp} assembles
the decorated blocks into a single permutation. We explain this at the end of this
section.

Here we isolate the general mechanism, in the setting of
sprout sequences whose seed is given by the exponential formula, as in
\cite[Example~1.1(i)]{ASS}.
Fix a sequence $c=(c_1,c_2,\dots)$ of nonnegative integers and, for
each $k\ge1$, a finite set $\mathcal{C}_k$ with
$\#\mathcal{C}_k=c_k$ (``connected structures of size $k$'').  A
\emph{$c$-structure} on a finite set $X$ is a pair
$(\pi,\gamma)$ where $\pi$ is a set partition of $X$ and $\gamma$
assigns to each block $B\in\pi$ an element
$\gamma(B)\in\mathcal{C}_{\#B}$.  By the exponential formula
\cite[Corollary~5.1.6]{EC2}, the number of $c$-structures on an
$n$-set is $n!\,[t^n]\,F(t)$, where
\[
F(t)\;=\;\exp\Bigl(\sum_{k\ge1}c_k\,\frac{t^k}{k!}\Bigr),
\]
and comparing with \eqref{eq:seedb} shows that the seed coefficients
of the sprout sequence $\mathfrak{R}$ with this seed are
\begin{equation}\label{eq:bexp}
  b_k\;=\;\frac{c_k}{(k-1)!}\,.
\end{equation}
Given a $c$-structure $(\pi,\gamma)$ on $[n]$ with $\ell$ blocks,
list the blocks as $B_1,\dots,B_\ell$ so that
$\max B_1<\dots<\max B_\ell$ (this ordering is unique), and define
the \emph{max-ordered type} of $(\pi,\gamma)$ to be the composition
\[
\tau(\pi)\;\coloneqq \;(\#B_1,\dots,\#B_\ell)\;\vDash\;n.
\]

\begin{proposition}\label{prop:expseed}
  With the notation above, let $\alpha\vDash n$ with partial sums
  $s_j$.  The number of $c$-structures on $[n]$ with max-ordered type
  $\alpha$ equals
  \[
  \prod_{j=1}^{\ell}\binom{s_j-1}{\alpha_j-1}\,c_{\alpha_j}
  \;=\;n!\,\prod_{j=1}^{\ell}\frac{b_{\alpha_j}}{s_j}
  \;=\;n!\;[\Psi^\alpha]\,\bR_n\,,
  \]
  where $\bR_n=\phi_b(S_n)$ is the canonical lift of
  Definition~\ref{def:lift} with $b$ as in \eqref{eq:bexp}.
  Consequently,
  \[ n!\,\bR_n=\sum_{\alpha\vDash n}
  \#\{\text{$c$-structures on $[n]$ of max-ordered type $\alpha$}\}\,\Psi^\alpha, \]
  and applying $\chi$ recovers the
  interpretation that $n!\,[p_\lambda]R_n$
  counts the $c$-structures on $[n]$ whose multiset of block sizes is
  $\lambda$.
\end{proposition}

\begin{proof}
  A $c$-structure of max-ordered type $\alpha$ is exactly an ordered set
  partition $(B_1,\dots,B_\ell)$ of $[n]$ with $\#B_j=\alpha_j$ and
  increasing maxima, counted by
  $\prod_j\binom{s_j-1}{\alpha_j-1}$ by
  Lemma~\ref{lem:orderedpartitions}, together with an independent choice
  of $\gamma(B_j)\in\mathcal{C}_{\alpha_j}$ for each $j$, in
  $\prod_jc_{\alpha_j}$ ways.  For the second expression, apply
  Lemma~\ref{lem:telescope} with $d=1$:
  \[
  \prod_{j}\binom{s_j-1}{\alpha_j-1}c_{\alpha_j}
  =n!\prod_j\frac{c_{\alpha_j}}{s_j\,(\alpha_j-1)!}
  =n!\prod_j\frac{b_{\alpha_j}}{s_j},
  \]
  by \eqref{eq:bexp}.  The third expression is
  Theorem~\ref{thm:ncsprout}(i), since
  $\prod_js_j=\piu(\alpha)$.  The final statement follows by summing
  over $\alpha\sim\lambda$ as in Corollary~\ref{cor:recover}.
\end{proof}

For example, taking $c_k=1$ for $k$ in a fixed set
$S\subseteq\mathbb{Z}_{>0}$ and $c_k=0$ otherwise recovers a max-ordered
refinement of \cite[Proposition~2.5]{ASS}: the number of set
partitions of $[n]$, all of whose block sizes lie in $S$, whose blocks
listed by increasing maxima have sizes $\alpha_1,\dots,\alpha_\ell$,
is $\prod_j\binom{s_j-1}{\alpha_j-1}$ if every $\alpha_j\in S$, and
$0$ otherwise.  Taking $c_k$ to be the number of connected simple
graphs on $k$ labeled vertices refines the graph example of
\cite[Example~1.1(i)]{ASS}, and so on.

Theorem~\ref{thm:main} is the evident ``doubled'' variant of
Proposition~\ref{prop:expseed}: blocks have even sizes $2\alpha_j$,
the decorations of a $2k$-set are the $E_{2k-1}$ alternating words
beginning with the maximum (Lemma~\ref{lem:blockwords}), and the seed
is $\sec(\sqrt t\,)$ with $b_k=kE_{2k-1}/(2k)!$, in accordance with
Lemma~\ref{lem:telescope} for $d=2$.  What is special to the
alternating case, and constitutes the content of
Lemmas~\ref{lem:concat} and~\ref{lem:decomp}, is that the block
decomposition can be read off from a single permutation $w$: the cut
points are located by the records of $\hat w$.

\section{Examples}\label{sec:examples}
This section gathers concrete illustrations of the results and complete
tables of $N(\alpha)$ for $n\le5$.

\begin{remark}\label{rem:verification}
  The finite checks quoted here were carried out by exact exhaustive
  enumeration for the displayed ranges: all alternating permutations of
  $[2n]$ were generated, their odd-indexed subwords and record
  compositions were computed directly, and the resulting counts were
  compared with the closed formulas and tables.
\end{remark}

We begin by running the bijection of
Section~\ref{sec:count} in both directions on the running example.

\begin{example}\label{ex:worked}
  Let $w=7,2,5,4,8,3,10,6,9,1\in\Alt_{10}$ as in
  Example~\ref{ex:intro}, so $\rc(\hat w)=(2,1,2)$ with partial sums
  $(s_1,s_2,s_3)=(2,3,5)$.  The bijection of Lemma~\ref{lem:decomp}
  cuts $w$ at the doubled record positions, producing
  \[
  u^1=7,2,5,4,\qquad u^2=8,3,\qquad u^3=10,6,9,1,
  \]
  with blocks $B_1=\{2,4,5,7\}$, $B_2=\{3,8\}$, $B_3=\{1,6,9,10\}$.
  One checks directly that each $u^j$ is alternating and begins with
  $\max B_j$; and $\max B_1=7<\max B_2=8<\max B_3=10$.  Conversely, the
  number of $w\in\Alt_{10}$ with $\rc(\hat w)=(2,1,2)$ is, by
  Theorem~\ref{thm:main},
  \[
  \binom{3}{3}E_3\cdot\binom{5}{1}E_1\cdot\binom{9}{3}E_3
  =2\cdot5\cdot 168=1680,
  \]
  in agreement with Table~\ref{tab:n5}.
\end{example}

\begin{example}\label{ex:tables}
  Tables~\ref{tab:small} and~\ref{tab:n5} list $N(\alpha)$ for all
  $\alpha\vDash n$, $2\le n\le5$ (for $n=1$ the only composition is
  $(1)$, with $N(1)=E_2=1$).  Every entry was confirmed by exhaustive
  enumeration of $\Alt_{2n}$ (Remark~\ref{rem:verification}).  For each
  $n$ the entries sum to $E_{2n}$, namely $5$, $61$, $1385$, $50{,}521$
  for $n=2,3,4,5$, respectively.

  \begin{table}[ht]
    \centering
    \begin{tabular}{|l|l|l|l|l|l|}
      \hline
      \multicolumn{2}{|c|}{$n=2$} & \multicolumn{2}{c|}{$n=3$} &
      \multicolumn{2}{c|}{$n=4$}\\
      \hline
      $\alpha$ & $N(\alpha)$ & $\alpha$ & $N(\alpha)$ & $\alpha$ & $N(\alpha)$\\
      \hline
      $(2)$   & $2$ & $(3)$     & $16$ & $(4)$       & $272$\\
      \hline
      $(1,1)$ & $3$ & $(2,1)$   & $10$ & $(3,1)$     & $112$\\
      \hline
        &     & $(1,2)$   & $20$ & $(1,3)$     & $336$\\
      \hline
        &     & $(1,1,1)$ & $15$ & $(2,2)$     & $140$\\
      \hline
        &     &           &      & $(2,1,1)$   & $70$\\
      \hline
        &     &           &      & $(1,2,1)$   & $140$\\
      \hline
        &     &           &      & $(1,1,2)$   & $210$\\
      \hline
        &     &           &      & $(1,1,1,1)$ & $105$\\
      \hline
    \end{tabular}
    \medskip
    \caption{$N(\alpha)=\#\{w\in\Alt_{2n}:\rc(\hat w)=\alpha\}$ for
    $n=2,3,4$.}
    \label{tab:small}
  \end{table}

  \begin{table}[ht]
    \centering
    \begin{tabular}{|l|l|l|l|}
      \hline
      $\alpha$ & $N(\alpha)$ & $\alpha$ & $N(\alpha)$\\
      \hline
      $(5)$     & $7936$ & $(2,1,2)$     & $1680$\\
      \hline
      $(4,1)$   & $2448$ & $(1,2,2)$     & $3360$\\
      \hline
      $(1,4)$   & $9792$ & $(3,1,1)$     & $1008$\\
      \hline
      $(3,2)$   & $2688$ & $(1,3,1)$     & $3024$\\
      \hline
      $(2,3)$   & $4032$ & $(1,1,3)$     & $6048$\\
      \hline
      $(2,2,1)$ & $1260$ & $(2,1,1,1)$   & $630$\\
      \hline
          &        & $(1,2,1,1)$   & $1260$\\
      \hline
          &        & $(1,1,2,1)$   & $1890$\\
      \hline
          &        & $(1,1,1,2)$   & $2520$\\
      \hline
          &        & $(1,1,1,1,1)$ & $945$\\
      \hline
    \end{tabular}
    \medskip
    \caption{$N(\alpha)$ for all sixteen compositions of $n=5$; the sum is
    $E_{10}=50{,}521$.}
    \label{tab:n5}
  \end{table}

  Grouping the $n=4$ column of Table~\ref{tab:small} by underlying
  partition illustrates Corollary~\ref{cor:recover}:
  \begin{align*}
    |\varphi(4)|=272,\quad
    |\varphi(3,1)|=112+336=448,\quad
    |\varphi(2,2)|=140, \\
    |\varphi(2,1,1)|=70+140+210=420,\quad
    |\varphi(1,1,1,1)|=105,
  \end{align*}
  in agreement with \eqref{eq:phiintro}; e.g.\
  $|\varphi(3,1)|=8!\cdot\frac{E_5}{6!}\cdot\frac{E_1}{2!}
  =40320\cdot\frac{16}{720}\cdot\frac12=448$.
\end{example}

\begin{example}\label{ex:lifts}
  By Theorem~\ref{thm:nsym} and Table~\ref{tab:small}, we have
  \[
  4!\,\bA_2=3\,\Psi^{(1,1)}+2\,\Psi^{(2)},
  \qquad
  6!\,\bA_3=15\,\Psi^{(1,1,1)}+20\,\Psi^{(1,2)}
  +10\,\Psi^{(2,1)}+16\,\Psi^{(3)}.
  \]
  Converting to the $S$-word basis via
  $\Psi_1=S_1$, $\Psi_2=2S_2-S_1^2$,
  $\Psi_3=3S_3-S_2S_1-2S_1S_2+S_1^3$ (from \eqref{eq:newton}) gives
  \[
  4!\,\bA_2=S^{(1,1)}+4\,S^{(2)},
  \qquad
  6!\,\bA_3=S^{(1,1,1)}+8\,S^{(1,2)}+4\,S^{(2,1)}+48\,S^{(3)}.
  \]
  Applying $\chi$ (so $S^\alpha\mapsto h_{\lambda(\alpha)}$) recovers
  the $h$-expansions $4!A_2=h_1^2+4h_2$ and
  $6!A_3=h_1^3+12h_2h_1+48h_3$ tabulated in \cite[\S5]{ASS}. Our
  computations likewise reproduce the tabulated expansions of
  $8!A_4$ and $10!A_5$.
\end{example}

\begin{example}\label{ex:negative}
  In the $S$-word and ribbon bases (Remark~\ref{rem:negative}), we have
  \begin{align*}
    8!\,\bA_4&=S^{(1,1,1,1)}+12\,S^{(1,1,2)}-8\,S^{(1,2,1)}
    +20\,S^{(2,1,1)}+16\,S^{(2,2)}\\
    &\qquad+192\,S^{(1,3)}+64\,S^{(3,1)}+1088\,S^{(4)},\\[2pt]
    8!\,\bA_4&=R_{(1,1,1,1)}+13\,R_{(1,1,2)}-7\,R_{(1,2,1)}
    +21\,R_{(2,1,1)}+49\,R_{(2,2)}\\
    &\qquad+197\,R_{(1,3)}+77\,R_{(3,1)}+1385\,R_{(4)}.
  \end{align*}
  Applying $\chi$ to the first line, the coefficients of the $S$-words
  with $\lambda(\alpha)=(2,1,1)$ sum to $12-8+20=24$, and those with
  $\lambda(\alpha)=(3,1)$ sum to $192+64=256$, matching the
  $h$-expansion
  \[ 8!A_4=h_1^4+24h_2h_1^2+256h_3h_1+16h_2^2+1088h_4 \]
  of \cite[\S5]{ASS}. The individual signed refinements, however, show
  that the nonnegative $h$-coefficients do not split nonnegatively along
  $S$-words of the canonical lift.
\end{example}

\begin{example}[An exponential seed]\label{ex:graphs}
  Take $c_k$ to be the number of connected simple graphs on $k$ labeled
  vertices, so $(c_1,\dots,c_5)=(1,1,4,38,728)$.
  Proposition~\ref{prop:expseed} then counts all labeled graphs on $[n]$
  by the sizes of their connected components listed in increasing order
  of their largest vertex.  For $n=5$ and $\alpha=(1,3,1)$, for
  instance, the count is
  $\binom{0}{0}c_1\cdot\binom{3}{2}c_3\cdot\binom{4}{0}c_1
  =1\cdot12\cdot1=12$; summing over all $\alpha\vDash5$ returns
  $2^{\binom52}=1024$.
\end{example}

\section{Appendix: AxiomProver's autonomous Lean verification}\label{sec:AI}
Here we provide the context for this project as well as the protocol used for Lean
formalization and verification (see \cite{Mathlib2020, Lean}).

\subsection*{Process}
The formal proofs provided in this work were developed and verified using Lean~4.31.0.
Compatibility with earlier or later versions is not guaranteed due to the
evolving nature of the Lean 4 compiler and its core libraries.
The relevant files are all posted in the following repository:
\begin{center}
  \url{https://github.com/AxiomMath/record-compositions}
\end{center}
For the process, there is exactly one input file: \texttt{task.md},
which contains the theorem statements for all three of the main theorems.
We note that no solution was provided to AxiomProver,
meaning AxiomProver had to generate the proofs by itself
in addition to formalizing them in Lean.

Based on this input file, AxiomProver generated two output files:
\begin{itemize}
  \item \texttt{problem.lean}, a Lean 4.31.0 formalization of the problem statement; and
  \item \texttt{solution.lean}, a complete Lean 4.31.0 formalization of the proof.
\end{itemize}

The human authors wrote this paper (without the use of AI) for human readers.
At first glance, the proofs found by AxiomProver may not resemble the narrative presented in this paper.
Turning a Lean file into a human-readable proof is difficult
because Lean is written as code for a type-checker.

\end{document}